\documentclass[11pt,a4paper]{article}

\usepackage{amsmath, amscd, amssymb, amsthm, latexsym}

\newcommand{\nat}{{\bf N }}
\newcommand{\Znat}{{\bf Z }}
\newcommand{\Qnat}{{\bf Q }}
\newcommand{\leqct}{\mbox{$\leq_{\rm ct}$}}

\newcommand{\eqct}{\mbox{$=_{\rm ct}$}}

\newcommand{\infct}{\mbox{$<_{\rm ct}$}}
\newcommand{\supct}{\mbox{$>_{\rm ct}$}}

\newcommand{\couple}[2]{\mbox{$\langle #1,#2 \rangle$}}
\newcommand{\segment}{\!\upharpoonright \!\!}

\newcommand{\jsl}{1}

\newcommand{\prog}{\{0,1\}^*}

\def\tte{\mathtt{e}} 
\def\ttp{\mathtt{p}}

\theoremstyle{plain}
\newtheorem{theorem}{Theorem}[section]
\newtheorem{definition-theorem}{Definition-Theorem}
\newtheorem{proposition}{Proposition}[section]
\newtheorem{corollary}{Corollary}[section]

\theoremstyle{definition}
\newtheorem{definition}{Definition}[section]
\newtheorem{example}{Example}[section]

\theoremstyle{remark}
\newtheorem{notation}{Notation}[section]
\newtheorem{remark}{Remark}[section]
\title{\large  \sc  
                        Church, Cardinal and Ordinal 
               \\       Representations of Integers
               \\        and Kolmogorov complexity
	       \\\bigskip Denis Richard's 60th Biirthday Conference}
\author{\small\sc Marie Ferbus-Zanda\\
{\footnotesize Universit\'e Paris 7}\\
{\footnotesize 2, pl. Jussieu 75251 Paris Cedex 05}\\
{\footnotesize France}\\
{\footnotesize\tt ferbus@logique.jussieu.fr}
\and {\small\sc Serge Grigorieff}\\
{\footnotesize LIAFA, Universit\'e Paris 7}\\
{\footnotesize 2, pl. Jussieu 75251 Paris Cedex 05}\\
{\footnotesize France}\\
{\footnotesize\tt seg@liafa.jussieu.fr}}
\begin{document}
\date{Mai 2002}
\maketitle
{\footnotesize \textnormal \tableofcontents}
%
\begin{abstract}
%
We consider classical representations of integers:
Church's function iterators, 
cardinal equivalence classes of sets,
ordinal equivalence classes of totally ordered sets.
Since programs do not work on abstract entities and require
formal representations of objects, we effectivize these abstract
notions in order to allow them to be computed by programs. 
To any such effectivized representation is then associated a
notion of Kolmogorov complexity.
We prove that these Kolmogorov complexities form a strict hierarchy
which coincides with that obtained by relativization to jump
oracles and/or allowance of infinite computations.
\end{abstract}

%
\section{Kolmogorov complexities}
%
%
We shall use the following notations.
\begin{notation}\label{not}$\\$
1) Inequality, strict inequality and equality up to a constant
between functions $\nat\to\nat$  are denoted as follows:
\medskip\\
$\begin{array}{rclcr}
f\ \leqct\ g &\Leftrightarrow &\exists c\ \forall n\ 
                                       (f(n)\leq g(n)+c) &&\\
f\ \infct\ g &\Leftrightarrow & 
           f \leqct g\ \wedge\ \forall c\ \exists n\ (f(n) < g(n)-c)
&&\\
f\ \eqct\ g &\Leftrightarrow & \exists c\ \forall n\ 
                                  (|f(n) - g(n)| \leq c)
          &\Leftrightarrow & f \leqct g\ \wedge\ g \leqct f
\end{array}$
\medskip\\
2) $Y^X$ (resp. $X\to Y$) denotes the set of total (resp. partial)
functions from $X$ into $Y$. 
\medskip\\
3) We denote ${\varphi}_{\tt e}$ the partial recursive function 
$\nat\to\nat$ with code ${\tt e}$.
\medskip\\
4) We denote $card(X)$ the number of elements of $X$ in
case $X$ is a finite set.
\end{notation}
%
\subsection{Kolmogorov complexity}
%
\begin{definition}[Kolmogorov, 1965 \cite{kolmo65}]\label{def:K}
$\\$Kolmogorov complexity $K:\nat\to\nat$ is defined as follows:
\begin{itemize}
\item[]  $K(n)$ {\em is the shortest length 
          of a program which halts and outputs $n$}
\end{itemize}
\end{definition}
\noindent 
To make Def. \ref{def:K} meaningful, some points have to be precised:
\begin{enumerate}
\item[(Q1)] {\em Where are programs taken from? In which alphabet?}
        \\ Bigger the alphabet, shorter the programs. 
	We shall therefore fix the alphabet of all programming
	languages to be binary.
        Now, Kolmogorov's invariance theorem insures that there exist
        optimal universal programming languages $U$ such that, 
	for any programming language $V$, the associated 
	complexity functions $K_U, K_V$ satisfy	$K_U \leqct K_V$ 
	(cf. Notations \ref{not}).
	In particular, if $U_1, U_2$ are two optimal universal 
	programming languages then $K_{U_1} \eqct K_{U_2}$.
\item[(Q2)]   {\em What does it mean that a program outputs  
        an integer $n$ ?}\\
        A program can only output a formal object such as a word
	in some alphabet which represents $n$. 
       However, there is again an invariance property relative to the
	usual representations of the output integer $n$ : 
	up to a constant, the same complexity functions are obtained
	when considering unary representation 
	or base $k$ representation for $k\geq 2$
	(cf. \cite{li-vitanyi} or \cite{ferbusgrigo} or \cite{shen}).
\end{enumerate}
The question aroused by $(Q2)$ is the core of this paper. 
We shall reconsider it in \S\ref{sec:rep} and \S\ref{sec:main}.
%
\subsection{Infinite computations and oracles}
%
Chaitin, 1976 \cite{chaitin76}, and Solovay, 1977 \cite{solovay77},
considered Kolmogorov complexity of infinite objects (namely
recursively enumerable sets) produced by infinite computations.
\\ 
Allowing programs leading to possibly infinite computations 
but finite output (i.e. remove the sole halting condition) , 
we get a variant of Kolmogorov complexity for which the results 
mentionned in $(Q1)$ above also apply.
\begin{definition}\label{def:Kvaria}$\\$
Allowing programs with possibly infinite computations, 
Kolmogorov complexity $K^{\infty} : \nat\to\nat$ is defined 
as follows:
\begin{itemize}
\item[]  $K^{\infty}(n)$ {\em is the shortest length 
       of a (possibly non halting) program
       which outputs $n$ in unary representation}
\end{itemize}
\end{definition}
\begin{remark}$\\$
The definition of $K^{\infty}(n)$ is dependent on the unary 
representation of outputs: $(Q2)$ does not apply.
\end{remark}
    
\ifnum 0=\jsl 
========={\bf Hors version Clermont}=========\\
$K^{\infty}$ has greater syntactical complexity than $K$.
\begin{proposition}  \label{prop:syntaxK}
$\\$ 
1) The relation $K(x)=y$ is $\Sigma^0_1\wedge\Pi^0_1$ but
neither $\Sigma^0_1$ nor $\Pi^0_1$.
\medskip\\
2) The relation $K^{\infty}(x)=y$ is $\Sigma^0_2\wedge\Pi^0_2$
but neither $\Sigma^0_2$ nor $\Pi^0_2$.
\end{proposition}
\begin{proof}
Point 1 is well-known. As for 2, observe that
\begin{eqnarray*}
K^{\infty}(x)=y  & \Leftrightarrow & 
\exists p\ (|p|=y\ \wedge\ \exists t\ \forall t'>t\ output(p,t')=x)
\\
&& \wedge\ \forall q\ (|q|<y\ \Rightarrow \forall t\ \exists t'>t\ 
                                              output(q,t')\neq x)
\end{eqnarray*}
{\bf TO DO.} prove that it is neither $\Sigma^0_2$ nor $\Pi^0_2$.
\end{proof}
===================================\\
\fi
Kolmogorov complexity can also be considered for computability
relative to some oracle.
\begin{definition}\label{def:Krelative}$\\$
Considering partial recursiveness relative to some oracle $A$,
Def. \ref{def:K}, \ref{def:Kvaria} lead to relative Kolmogorov 
complexities $K^A : \nat\to\nat$ and $K^{A,\infty} : \nat\to\nat$.
\\ In case $A=\emptyset^i$ is the $i$-th jump (i.e a 
$\Sigma^0_i$-complete set of integers), we simply write $K^i, K^{i,\infty}$. 
\\ For explicit values $i=1,2,\ldots$ we also write 
$K', K''\ldots, {K'}^{\infty}, {K''}^{\infty}\ldots$.
\end{definition}
Jump oracle $\emptyset'$ (resp. $\emptyset''$,\ldots) allows 
the computation to decide for free (in a single computation step)
any $\Sigma^0_1$ or $\Pi^0_1$ 
(resp. $\Sigma^0_2$ or $\Pi^0_2$,\ldots)
statement about integers.
As expected and is well known, such an oracle 
leads to an extended notion of programs. Which
allows to
\begin{itemize}
\item  compute non recursive sets and functions,
\item  get much shorter programs to compute finite objects
       or recursive sets and functions, namely 
       $K \supct K' \supct K'' \supct \ldots$.
\end{itemize}
Infinite computations also lead to shorter programs but,
as proved by Becher, 2001 \cite{alpha}, they do not help
as much as the jump oracle.
\begin{proposition}[\cite{alpha}]  \label{prop:Kcomparing}
$\\$\centerline{$K \supct K^{\infty} \supct K' 
                   \supct {K'}^{\infty} \supct K'' \supct \ldots$}
\end{proposition}
\ifnum 0=\jsl 
========={\bf REMARQUE A TRAVAILLER}=========\\
Unaire, binaire, 2-adique, Avizienis: pas pareil\\
On peut cependant passer du binaire a l'unaire:
Kunaire <= Kbinaire\\
===================================\\
such an oracleleads to more abstract programs. \\
Les calculs infinis permettent aussi de calculer de nouveaux 
ensembles et fonctions\\
Preuve de Prop. \ref{prop:Kcomparing}
a faire. Becher prouve ca pour les outputs en binaire
\\===================================
\fi
%
\subsection{Prefix Kolmogorov complexities}
%
Introduced by Levin \cite{levin73} and Chaitin \cite{chaitin75}
(cf. \cite{li-vitanyi}),
prefix complexity $H:\nat\to\nat$ is the analog of Kolmogorov
complexity obtained by restricting programming languages to be 
prefix-free: two distinct programs have to be incomparable 
with respect to the prefix ordering.
Variants $H', H''\ldots, H^{\infty}, {H'}^{\infty}, \ldots$
involving infinite computations and/or relativization
are defined in the obvious way.

As concerns all questions considered in this paper, everything 
goes through with straightforward changes. 
So that we shall deal exclusively with the $K$ complexity and its 
variants $K^{\infty}, K', {K'}^{\infty}, K'',\ldots$.
%
%
%
\section{Representations of integers}\label{sec:rep}
%
%
%
The purpose of the paper is to consider particular representations
of integers and to study their influence on the definition of 
Kolmogorov complexity as pointed in question $(Q2)$
relative to Def. \ref{def:K} above. This will lead to a 
strict hierarchy of Kolmogorov complexities which happens to 
coincide with that obtained in Prop. \ref{prop:Kcomparing}.
%
\subsection{Abstract representations of integers}
%
A representation of integers involves some abstract object $C$
(in practice much more complex than $\nat$ itself) such that
some of its elements characterize the diverse integers through some 
property.
Each representation illuminates some role and/or properties of 
integers.
\ifnum 0=\jsl 
========={\bf REMARQUE A TRAVAILLER}=========\\
Lien representation $<--->$ properties\\
Complexite logique des proprietes en jeu
\\===================================
\fi
\begin{definition}  \label{def:representation}$\\$
A representation of integers is a pair $(C,R)$ where 
$C$ is some (necessarily infinite) set and $R$ is a {\em surjective} 
partial function $R : C \rightarrow \nat$. 
\end{definition}
\begin{remark}$\\$
In practice, $Domain(R)$ will be a strict subset of $C$, in fact a 
very small part of $C$.
\end{remark}
\begin{example}\label{ex:rep1}$\\$
1) The unary representation of integers corresponds to the free 
algebra built up from one generator and one unary function,
namely $0$ and the successor function $x\mapsto x+1$.
\medskip\\
2) The various base $k$ (with $k\geq 2$) representations of integers
also involve term algebras, not necessarily free.
They differ by the sets of digits they use but all are based on the
usual interpretation 
$d_n \ldots d_1 d_0 \mapsto \sum_{i=0,\ldots,n} d_i k^i$ which, 
seen in Horner style:\\
\centerline{
$k(k(\ldots k (k d_n + d_{n-1}) + d_{n-2})\ldots) + d_1) + d_0$}
is a composition of applications 
$S_{d_0} \circ S_{d_1} \circ \ldots \circ S_{d_n}(0)$ where 
$S_d : x \mapsto kx+d$. 
If a representation uses digits $d\in D$ then it corresponds to 
the algebra generated by $0$ and the $S_d$'s where $d\in D$.  
\begin{enumerate}
\item[i.] The $k$-adic representation uses digits $1,2,\ldots,k$
      and corresponds to a free algebra built up from one generator 
      and $k$ unary functions.
\item[ii.] The usual $k$-ary representation uses digits 
      $0,1,\ldots,k-1$ and corresponds to the quotient
      of a free algebra built up from one generator and 
      $k$ unary functions, 
      namely $0$ and the $S_d$'s where $d=0,2,\ldots, k-1$,
      by the relation $S_0(0)=0$.
\item[iii.] Avizienis base $k$ representation uses digits 
      $-k+1,\ldots,-1,0,1,\ldots,k-1$, 
      (it is a much redundant representation used to perform 
      additions without carry propagation) and
      corresponds to the quotient of the free algebra built up from 
      one generator $a$ and $2k-1$ unary functions, 
      namely $0$ and the $S_d$'s where 
      $d=-k+1,\ldots,-1,0,1,\ldots, k-1$,
      by the relations 
      $\forall x\ (S_{-k+i}\circ S_{j+1}(x)=S_i\circ S_j(x))\ $ where
      $-k<j<k-1$ and $0<i<k$.
\end{enumerate}
3) $R:\nat^4\to\nat$ such that $R(x,y,z,t)=x^2+y^2+z^2+t^2$
is a representation based on Lagrange's four squares 
theorem  
\medskip\\
4) $R:Prime^{\leq 7}\to\nat$ such that 
$R(x_1,\ldots ,x_i)=x_1+\ldots+x_i$ (with $i\leq 7$)
is a representation based on Schnirelman's theorem 
(as improved by RamarŽ, 1995) which insures that every number is the 
sum of at most 7 prime numbers.
\end{example}
Besides such number theoretic representations of integers, 
we shall consider classical set theoretic representations
involving higher order objects.

\begin{example}\label{ex:rep2}$\\$
1) {\bf Church's representation.}$\\$
Integers are viewed as function iterators (which are type 2 objects)
$f \mapsto f^{(n)}$ where 
$f^{(0)}=Id$ and $f^{(n+1)} = f^{(n)} \circ f$. 
Thus, $C$ is the class containing all functional sets 
$(X\to X)^{X\to X}$ (cf. Notations \ref{not})
and $R$ is defined on the proper subclass of $C$ constituted
of functionals which are finite iterators on some $X\to X$
and $R(F)=n$ if and only if $F(f)=f^{(n)}$ for all $f:X\to X$.
\medskip\\
1bis) {\bf $\Znat$-Church's representation.}$\\$ 
Negative iterators $f \mapsto f^{(-n)}$ are defined as follows:
\\ \indent {\em i.}  $Domain(f^{(-n)}) = Range(f^{(n)})$
\\ \indent {\em ii.} $f^{(-n)}(x)=y$ if $y$ is the smallest 
                     such that $f^{(n)}(y)=x$ 
		     (relative to some fixed well-order on $X$) 
\\ The definitions of $C$ and $R$ are analog to that in point 1.
\medskip\\
2) {\bf Cardinal representation.}$\\$
Integers are viewed as equivalence classes of sets relative to 
cardinal comparison.
Thus, $C$ is the class of all sets and $R$ is defined on the proper
subclass of $C$ constituted of finite sets and $R(X)$ is the 
cardinal of the set $X$.
\medskip\\
3) {\bf $\Znat$-Cardinal representation.}$\\$
Relative integers are viewed as differences of natural integers 
which are themselves viewed via cardinal representation.
Thus, $C$ is the class of all pairs of sets and $R$ is defined on 
the proper subclass of $C$ constituted of finite sets and 
$R(X,Y)$ is the difference of the cardinals of $X$ and $Y$.
\medskip\\
4) {\bf Ordinal representation.}$\\$
Integers are viewed as equivalence classes of totally ordered sets.
Thus, $C$ is the class of totally ordered sets and $R$ 
is defined on the proper subclass of $C$ constituted of finite 
totally ordered sets and $R(X)$ is the order type of $X$.
\end{example}
\ifnum 0=\jsl 
========={\bf REMARQUES A TRAVAILLER}=========\\
Church et action de semi-groupe sur le semi-groupe des fonctions\\
Lien iterateurs et cardinaux (et ordinaux ???)\\
Tour d'iterateurs et limites de tours\\
{\bf $\Znat$-Cardinal} permet de comparer sans compter.
Non det des choix successifs quand on compare\\
{\bf $\Qnat$-Cardinal representation.} Rational are\\ 
Entiers de gauss\\
1, j et j2\\
Racines n-iemes de l'unite\\
===================================
\fi
%
\subsection{Formal representations of integers}
%
\ifnum 0=\jsl 
========={\bf REMARQUES A TRAVAILLER}=========\\
representation = choisie par rapport aux operations qu'on veut
utiliser (Ex: algo de Stein)\\
``effectivize" abstract sets: 
sets of sets of integers will be r.e. sets interpreted 
      as sets of $\equiv$ equivalence classes of codes of r.e. sets
      where $e\equiv e'\ \Leftrightarrow\ W_e = W_{e'}$,
\\ Rk: A CRITIQUER
{\em Moreover, programs do not deal with extensional
       sets, functions and functionals but with some intensional
       representations of these objects.}
\\===================================\\
\fi
A formal representation of an integer $n$ is a finite object
(in general a word) which describes some characteristic property 
of $n$ or some abstract object which characterizes $n$.
In fact, each particular representation is really a choice made
in order to access special operations or stress special 
properties of integers.\\
The computer science (or recursion theoretic) point of view brings
an objection to the consideration of abstract sets, functions and
functionals as we did in Example \ref{ex:rep2}:
\begin{itemize}
\item {\em We cannot apprehend abstract sets, functions and 
        functionals but solely programs to compute them 
	(if they are computable in some sense).}
\item {\em Moreover, programs dealing with sets, functions and 
       functionals have to go through some intensional 
       representation of these objects in order to be able to 
       compute with such objects.}
\end{itemize}
Thus, to get effectiveness, we turn from set theory to recursion 
theory and ``effectivize" abstract sets: 
\begin{itemize}
\item sets of integers will be recursively enumerable (r.e.), 
      i.e. domains of partial recursive functions,
\item functions on integers will be partial recursive,
\item functionals will be partial recursive in the sense of 
      higher type recursion theory (cf. usual textbooks
      \cite{rogers} or \cite{odifreddi}).
\end{itemize}
Though abstract representations are quite natural and
conceptually simple, their effectivized versions are quite complex:
the sets of programs computing objects in their domains 
involve levels $2$ or $3$ of the arithmetical hierarchy.
In particular, {\em such representations are not all
Turing reducible one to the other}.
In the sequel we shall only consider type $\leq 2$ 
representations. 
In order to get the adequate notion of recursion-theoretic 
representations of integers, we have to review some
higher order recursion concepts.
%
\subsection{Effectivization of sets of type $\leq 2$ objects}
           \label{sub:type2rec}
%
First, we recall the notion of type $2$ recursion that
we shall use.
\begin{definition}[Effective operations]
  \label{def:effectiveOperations}
$\\$Let $X,Y,Z,T$ be type $0$ spaces 
(i.e. $\nat$, $\nat^k$, $\{0,1\}$,\ldots).
We denote $PR(X\to Y)$ the set of partial recursive functions 
$X\to Y$.
\\
An effective operation 
         $F : PR(X\to Y) \to PR(Z\to T)$ 
or       $F : PR(X\to Y) \to Z$
is an operation which can be defined via partial recursive operations 
on the G\"odel numbers of partial recursive functions.
In other words, letting  $U^{PR(X\to Y)}$ denote a partial recursive 
enumeration of $PR(X\to Y)$, there exists $f$ such that 
the following diagram commutes 
         \[ \begin{CD}
             PR(X\to Y) @> F >>PR(Z\to T)\\
             @AU^{PR(X\to Y)}AA @AAU^{PR(Z\to T)}A \\
             \prog @> f >>\prog
            \end{CD} \]
We denote $Eff((X\to Y) \to (T\to Z))$ (resp. $Eff((X\to Y) \to Z))$
the family of effective operations from $PR(X\to Y)$ 
into $PR(Z\to T)$ (resp. into $Z$).  
\end{definition}
Let's recall the following fact.
\begin{theorem}[Myhill \& Shepherdson, 1955]
$\\$Effective operations \\
\centerline{$F:PR(X\mapsto Y) \to Z$ \ 
                 (or $F:PR(X\mapsto Y) \to PR(Z\mapsto T)$)}
are exactly the restrictions of effectively continuous functionals
\\ \centerline{$F: Y^X \to Z$ \ (or $F: Y^X \to T^Z$)}
in the sense of Uspenskii, 1955
(cf. Rogers \cite{rogers}, or Odifreddi \cite{odifreddi}).
\end{theorem}
\begin{definition}[Effectivization of higher type sets]  
\label{def:effectivization}
$\\${\em 1)} The effectivization of the type $1$ sets $Y^X$ and
$X\to Y$ is the subset $PR(X\to Y)$ of partial recursive functions.
\\ {\em 2)} The effectivization of the type $2$ set $(T^Z)^{Y^X}$ 
(resp. $Z^{Y^X}$) is the subset of effective operations 
$PR(X\to Y) \to PR(Z\to T)$ (resp. $PR(X\to Y) \to Z$).
\end{definition}
We can now define {\em strongly universal} partial recursive
functions and {\em strongly universal} effective operations.
\begin{definition}[Universal enumerations] \label{def:universal1}
$\\$Let $X,Y,Z,T$ be type $0$ sets and let ${\cal E}$ be 
$X$ or $PR(X \to Y)$ (resp. $Eff((X\to Y)\to (Z\to T))$).
\medskip\\
{\em 1)} A partial recursive function (resp. effective operation)
$U : \prog \rightarrow {\cal E}$ is {\em universal} 
if there is a recursive function 
${\tt comp}:\prog \times \prog \to \prog$ 
such that if we denote $U_{\tte}$ the function such that
$U_{\tte}(\ttp)=U({\tt comp}(\tte,\ttp))$,
then the family $(U_{\tte})_{\tte\in \prog}$ is an enumeration of 
the class of partial recursive functions (resp. effective operations)
from $\prog$ to ${\cal E}$.
\\ {\em Intuition}: Words in $\prog$ are seen as programs.
A partial recursive function (resp. effective operation) 
$\prog\to {\cal E}$ maps a program to 
the object it computes (which lies in ${\cal E}$).
Function $\ttp\mapsto {\tt comp}(\tte,\ttp)$ is therefore
seen as a compiler.
\\ We say that $\tte$ is a G\"odel number for 
$F\in {\cal E}$ if $F=U_{\tte}$.
\medskip\\
{\em 2)} $U$ is {\em strongly universal} if it is universal and  
for each index $\tte$, there is a constant $c(\tte)$ such that 
for all $\ttp \in \prog$, we have \\ \centerline{
$|{\tt comp}(\tte,\ttp)| \leq |\ttp|+c(\tte)$}
\end{definition}
The following theorem is a classical result of recursion theory
which is crucial for the definition of Kolmogorov complexity
in \S\ref{sub:Krep}.
\begin{theorem} \label{thm:universal1}
$\\$There exists a strongly universal partial recursive function 
(resp. effective operation).\\
Moreover, one can suppose that ${\tt comp}$ is the pairing function
$\couple{}{}$ defined by\\
\centerline{$\couple{\epsilon}{\ttp} = 1\ttp$\ , \
$\couple{\tte_1\tte_2 \cdots \tte_n}{\ttp} = 
         0\tte_10\tte_2 \cdots 0\tte_n1\ttp$}
which satisfies the equality \\ \centerline{
$|\couple{\tte}{\ttp}| = |\ttp| + 2\ |\tte|+ 1$}
\end{theorem}
%
\subsection{Recursion-theoretic representations of integers}
%
Finally, we are in a position to introduce the wanted definition.
\begin{definition}[Recursion-theoretic representations]
                   \label{def:recursion theoreticalrep}
$\\$A recursion theoretical representation of \nat 
(resp. \Znat) is any {\em surjective} function 
   $\rho : C \rightarrow \nat$ (resp. $\rho : C\to \Znat$) 
where $C$ is the effectivization ${\cal E}$ of some 
higher type set.
\end{definition}
We use Example \ref{ex:rep2} to illustrate the effectivization 
processes described in \S\ref{sub:type2rec}.
\begin{example}\label{ex:rep2effectivized}$\\$
1) {\bf Effective Church and $\Znat$-Church representations.}\\
Iterators $It_{n} : PR(\nat\to\nat)\to PR(\nat\to\nat)$
of partial recursive functions are inductively 
defined as follows for $n\in \nat$ : 
\\ \indent {\em i.}   $It_{0}(f)=f$
\\ \indent {\em ii.}  $It_{n+1}(f)=It_{n}(f) \circ f$
\medskip\\
Negative iterators $It_{-n} : PR(\nat\to\nat)\to PR(\nat\to\nat)$
are defined as follows: 
\\ \indent {\em iii.}  $It_{-n}(f)$ has domain $Range(It_{n}(f)$
\\ \indent {\em iv.} $It_{-n}(f)(x)=y$ if $It_{n}(f)(y)=x$ and
$It_{n}(f)(y)$ is the first halting computation among all
computations $It_{n}(f)(0), It_{n}(f)(1), \ldots$.
\medskip\\
We let
$Church^{\nat} : Eff((\nat \to \nat) \to (\nat \to \nat)) \to \nat$
be so that
\\ \indent $Church^{\nat}(F) = n$ 
               if $F=It_{n}$ for some $n\in \nat$ , \ 
               otherwise undefined
\medskip\\
and 
$Church^{\Znat} : Eff((\nat \to \nat) \to (\nat \to \nat)) \to \Znat$
be so that
\\ \indent $Church^{\Znat}(F) = n$ 
               if $F=It_{n}$ for some $n\in \Znat$ , \ 
               otherwise undefined
\medskip\\
2) {\bf Effective cardinal and $\Znat$-cardinal representations.}\\ 
We let $card^{\nat} : PR(\nat \to \nat) \to \nat$
be so that
\\ \indent {\em i.} $card^{\nat}(f)$ is defined if and only if 
              $domain(f)$ is finite
\\ \indent {\em ii.} $card^{\nat}(f)= card(domain(f))$ 
\medskip\\
We let 
$card^{\Znat} : PR(\nat \to \nat)\times PR(\nat \to \nat) \to \Znat$
be so that
\\ \indent {\em iii.} $card^{\Znat}(f,g)$ is defined if and only if 
              $f,g$ have finite domains
\\ \indent {\em iv.} $card^{\Znat}(f,g)=
                      card(domain(f)) - card(domain(g))$
\medskip\\
3) {\bf Effective ordinal representation.}\\
We let $ord^{\nat} : PR({\nat}^{2} \to \nat) \to \nat $
be so that
\\ \indent {\em i.} $ord^{\nat}(f)$ is defined if and only if 
              the quotient order associated to the transitive 
	      closure of $domain(f)$ is finite
\\ \indent {\em ii.} $ord^{\nat}(f)$ is the order type of this
             quotient order.
\end{example}
The following result measures the syntactical complexity of 
the domain and the graph of the functionals 
$Church^{\nat}, Church^{\Znat}, card^{\nat}, card^{\Znat}, ord$
in terms of the associated index sets.
\begin{proposition}[Syntactical complexity of representations]
    \label{prop:syntaxcomplex}
$\\$
1) {\bf Church representations.}$\\$
The set of pairs $(\tte,n)$ such that $n\in \nat$ 
(resp. $n\in\Znat$) and $\tte$ is a G\"odel number for the iteration
functional $It_{n}$ is $\Pi_{2}^0$-complete.
\\ The set of G\"odel numbers of effective functionals 
in the domain of $Church^{\nat}$ (resp. $Church^{\Znat}$) 
is $\Sigma_{3}^0$-complete. 
\medskip\\
2) {\bf Cardinal representations.}$\\$ 
The set of pairs $(\tte,n)$ such that $n\in \nat$ 
(resp. $n\in\Znat$) and $\tte$ is the G\"odel number of a
partial recursive function the domain of which is finite with $n$ 
elements is $\Sigma_{2}^0$-complete.
\\ The set of G\"odel numbers of partial recursive functions
with finite domains is $\Sigma_{2}^0$-complete. 
\medskip\\
3) {\bf Ordinal representation.}$\\$
The set of pairs $(\tte,n)$ such that $n\in \nat$ 
and $e$ is the G\"odel number of a partial recursive function 
$\nat^2\to\nat$ such that the quotient order associated to the 
transitive closure of $domain(f)$ is finite with $n$ elements is
$\Sigma_{3}^0$-complete.
\\ The set of G\"odel numbers of partial recursive functions
such that the quotient order associated to the transitive 
closure of $domain(f)$ is finite is $\Sigma_{3}^0$-complete. 
\end{proposition}
%
\subsection{Kolmogorov complexity and representations of integers}
           \label{sub:Krep}
%
The usual definition of Kolmogorov complexity can be extended
to any recursion-theoretic representation of integers.
\begin{definition}[Kolmogorov complexity relative to 
                       a representation] \label{def:KChurchKcard}
Let ${\cal E}$ be the effectivization  of some 
higher type set and let $\rho : {\cal E} \rightarrow \nat$ 
(resp. $\rho : {\cal E} \rightarrow \Znat$)
be a recursion-theoretic representation of integers.
\\
Considering the diagram 
$$\{0,1\}^* \stackrel{U^{\cal E}}{\longrightarrow} 
   {\cal E} \stackrel{\rho}{\longrightarrow} \nat$$
where $U^{\cal E}$ is some strongly universal enumeration of
${\cal E}$, Kolmogorov complexity $K_{\rho} : \nat \to \nat$ 
is defined as
$$K_{\rho}^{\nat}(n)=\min\{|p| \ : \rho(U^{\cal E}(p))=n\}$$
The definition of $K_{\rho} : \Znat \to \nat$ is analog.
\end{definition}
Theorem \ref{thm:universal1} implies an invariance theorem
for strongly universal enumerations. 
Which insures that the above definition does not depend 
(up to a constant) of the particular choice of the strongly 
universal enumeration $U^{\cal E}$ of ${\cal E}$.
\begin{remark}
The domain of $\rho\circ U^{\cal E}$ is not
recursively enumerable in general 
(cf. Prop. \ref{prop:syntaxcomplex}).
\end{remark}
%
%
\section{Main Theorem}\label{sec:main}
%
%
\subsection{Main theorem}
%
Reconsidering the answer to $(Q2)$ given after Def. \ref{def:K}, 
the main theorem of this paper insures the following. 
\begin{itemize}
\item  Kolmogorov complexity is much dependent on the chosen
       higher order effective representations of integers,
\item  The Kolmogorov complexities associated to representations
       of Example \ref{ex:rep2effectivized} constitute a hierarchy
       which coincides with that obtained with infinite computations
       and relativization to the jumps 
       (cf. Prop. \ref{prop:Kcomparing}).
\end{itemize}
Thus, Kolmogorov complexity measures the complexity of 
higher order effective representations and allows to 
classify them.
\begin{theorem}[Main Theorem]\label{thm:main}$\\$
Let 
$K^{\nat}_{Church}, K^{\Znat}_{Church}, K^{\nat}_{card}, 
                                        K^{Z\nat}_{card}, K_{ord}$ 
be the Kolmogorov complexities associated to the effective versions
of the higher order representations described in Example 
\ref{ex:rep2effectivized}. 
Then
$$K^{\nat}_{Church} \eqct K^{\Znat}_{Church}\segment\nat 
                  \eqct K$$
$$K^{\nat}_{card}   \eqct K^{\infty}$$
$$K^{\Znat}_{card}\segment\nat \eqct K'$$
$$K^{\nat}_{ord}  \eqct {K'}^{\infty}$$
So that we have
$$\ K^{\nat}_{Church}\ \eqct\ K^{\Znat}_{Church}\segment\nat\
\supct\ K^{\nat}_{card}\ \supct\ K^{\Znat}_{card}\segment\nat\
\supct\ K_{ord}$$
\end{theorem}
%
\subsection{Representations $\rho$ such that $K_{\rho} = K$}
%
The following theorem gives simple sufficient conditions on
representations $\rho$ so that the associated Kolmogorov complexity 
$K_{\rho}$ be equal, up to a constant, 
to usual Kolmogorov complexity $K$.\\
In particular, these conditions will apply to Church's 
representations.
\begin{theorem}  \label{thm:prtype1}
$\\$Let ${\cal E}$ be $PR(X\to Y)$ or $Eff((X\to Y)\to (Z\to T)$ 
and $U^{\cal E} : \prog \to {\cal E}$ be strongly universal.
Let $\rho : {\cal E} \to \nat$ (resp. $\rho : {\cal E} \to \Znat$)
be a representation of integers.
\\ Consider the following conditions:
\begin{itemize}
\item[(*)] $\rho \circ U^{\cal E}$ is the restriction to its domain
           of some partial recursive function $f : \prog \to \nat$
	   ($f : \prog \to \Znat$). 
\item[(**)] $\rho$ is {\em effectively surjective}:
            there exists a {\em total recursive} function 
            $g : \nat \to \prog$ (resp. $g : \Znat \to \prog$)
	    such that 
             $\rho(U^{\cal E}(g(n)))=n$ for all $n\in\nat$
	     (resp.  $n\in\nat$)
\end{itemize}
1) Condition (*) implies $K \leqct K_{\rho}$.
\medskip\\ 
2) Condition (**) implies $K_{\rho} \leqct K$.
\end{theorem}
\begin{proof}$\\$
1) Let $n\in\nat$ and let $p\in\prog$ be a minimal length program
such that $\rho(U^{\cal E}(p)) = n$. Then $K_{\rho}(n) = |p|$. 
Condition (*) implies $f(p) = n$. Therefore, viewing $f$ as a 
programming language and using Kolmogorov's invariance theorem 
there is a constant $c$ independent of $n$ such that\\
\centerline{$K(n)\leq K_f(n) + c \leq |p|+c = K_{\rho}(n)+c$}
Thus, $K \leqct K_{\rho}$.
\medskip\\
2) Let $U^{\nat}:\prog\to\nat$ be a strongly universal enumeration
of $\nat$. The strong universality of $U^{\cal E}$ insures that 
there exists ${\tt e}$ such that 
$$\forall p\ \ U^{\cal E}(g(U^{\nat}(p))) 
            = U^{\cal E}(\langle {\tt e}, p\rangle)$$
Let $n\in\nat$ and let $p\in\prog$ be a minimal length program
such that $U^{\nat}(p) = n$. Then $K(n) = |p|$. 
\\ Condition (**) implies 
           $\rho(U^{\cal E}(g(U^{\nat}(p)))) = n$.
Thus, $\rho(U^{\cal E}(\langle {\tt e}, p\rangle)) = n$, 
whence (using Thm \ref{thm:universal1})
    $$K_{\rho}(n) \leq |\langle {\tt e}, p\rangle|
                   = |p|+2|{\tt e}|+1 = K(n)+2|{\tt e}|+1$$ 
and therefore $K_{\rho} \leqct K$.
\end{proof}
\begin{corollary}$\\$
\centerline{$K_{Church}^{\nat} 
         \eqct K_{Church}^{\Znat}\segment\nat \eqct K$}
\end{corollary}
\begin{proof}$\\$
We show that conditions (*) and (**) of Theorem \ref{thm:prtype1}
are satisfied for the $\rho$ associated to $K_{Church}^{\nat}$.
The argument also applies with straightforward modifications to
$K_{Church}^{\Znat}\segment\nat$.\\
To get condition (**), just design a program for functional $It_n$.
\\
As for condition (*), let 
${\cal E}={Eff((\nat \to \nat) \to (\nat \to \nat))}$ and
$Succ:\nat\to\nat$ be the successor function and 
define $f\prog\to\nat$ as follows: 
           $$f(p) = U^{\cal E}(p)(Succ)(0)$$     
If $\rho(U^{\cal E}(p))=n$ then $U^{\cal E}(p) = It_n$ so that
$U^{\cal E}(p)(Succ)$ is the function $x\mapsto x+n$ and
$f(p)=U^{\cal E}(p)(Succ)(0)=n$.
\end{proof}
\begin{remark}
$\\$1) Conditions (*) and (**) both hold trivially in case
$\rho \circ U^{\cal E}$ is a partial recursive function 
or an effective functional.
\medskip\\
2) Theorem \ref{thm:prtype1} relativizes to computability
with an oracle.
\end{remark}
%
\subsection{$K^{\nat}_{card}$ and $K^{\infty}$}
%
\begin{theorem} \label{thm:Kcard}$\\$
\centerline{$K_{card}^{\nat} \eqct K^{\infty}$}
\end{theorem}
\begin{proof}$\\$
1) ${\cal E}=PR(\nat \to \nat)$.
Let $n\in\nat$ and let $p\in\prog$ be a minimal length program
which outputs $n$ in unary through a possibly infnite computation.
Then $K^{\infty}(n) = |p|$. 
Define $h:\prog\to\prog$ such that $h(p)$ is the following program
for a partial recursive function $\varphi_{h(p)}:\nat\to\nat$ :\\
$$\varphi_{h(p)}(t) = \mbox{{\tt IF} at step } t 
                     \mbox{ program } p \mbox{ outputs } 1
                     \mbox{ {\tt THEN} } 1
                     \mbox{ {\tt ELSE} undefined}$$
Clearly, $card^{\nat}(U^{\cal E}(h(p)) = n$.
The strong universality of $U^{\cal E}$ insures that 
there exists ${\tt e}$ such that 
$$\forall p\ \ U^{\cal E}(h(p))
             = U^{\cal E}(\langle {\tt e}, p\rangle)$$ 
This leads to $K^{\nat}_{card}(n) \leq K^{\infty}(n)+2|{\tt e}|+1$,
whence $K^{\nat}_{card} \leqct K^{\infty}$.
\medskip\\
2) Let $n\in\nat$ and let $p\in\prog$ be a minimal length program
such that $card^{\nat}(U^{\cal E}(p)) = n$.
Define $h:\prog\to\prog$ such that $h(p)$ behaves as follows:
\begin{itemize}
\item  $h(p)$ emulates some dovetailing of all computations 
       $U^{\cal E}(p)(i)$ for $i=0,1,2,\ldots$,
\item  each time some computation $U^{\cal E}(p)(i)$ halts
       (i.e. a new point $i$ is proved to be in the domain of
       $U^{\cal E}(p)$) then $h(p)$ output $1$.
\end{itemize}
It is clear that $h(p)$ outputs $n$ in unary.
Thus, $K^{\infty}(n) \leq |h(p)|$. But there is some constant $c$
such that $|h(p)|= |p|+c$, whence $K^{\infty}(n) \leq |p|+c$ and
therefore $K^{\infty} \leqct K^{\nat}_{card}$.	 
\end{proof}
%
\subsection{$K^{\Znat}_{card}\segment\nat$ and $K'$}
%
%
\begin{theorem}$\\$
\centerline{$K^{\Znat}_{card}\segment\nat \eqct K'$}
\end{theorem}
\begin{proof}$\\$
Now ${\cal E}=(PR(\nat \to \nat))^2$ and $U^{\cal E}$ is a pair
of functions $(U^{\cal E}_1, U^{\cal E}_2)$.
\\
1) Let $n\in\nat$ and let $p\in\prog$ be a minimal length program
such that 
$card^{\Znat}(U^{\cal E}(p)) 
= card^{\nat}(U^{\cal E}_1(p)) - card^{\nat}(U^{\cal E}_2(p)) = n$.
Define $h:\prog\to\prog$ such that $h(p)$ is a program which 
uses oracle $\emptyset'$ and behaves as follows:
\begin{itemize}
\item  $h(p)$ emulates some dovetailing of all computations 
       $U^{\cal E}_1(p)(i), U^{\cal E}_2(p)$ for $i=0,1,2,\ldots$,
\item  At each computation step, $h(p)$ asks the oracle whether 
       there is still some computation that will halt.
       If the answer is ``NO" then $h(p)$ halts and outputs
       the difference of the number of points $i$'s on which
       $U^{\cal E}_1(p)$ has been checked to converge and that for
       $U^{\cal E}_2(p)$
\end{itemize}
It is clear that $h(p)$ outputs $n$. 
Thus, $K'(n) \leq |h(p)|$. But there is some constant $c$
such that $|h(p)|= |p|+c$, whence $K'(n) \leq |p|+c$ and
therefore $K' \leqct K^{\Znat}_{card}\segment\nat$.
\medskip\\
2) Let $n\in\nat$ and let $p\in\prog$ be a minimal length program
using oracle $\emptyset'$ which outputs $n$ in unary. 
Then $K'(n) = |p|$. 
To emulate computations using oracle $\emptyset'$, we shall use
Chaitin's harmless overshoot technique \cite{chaitin76}:
\begin{itemize}
\item  If an $\exists x\ldots$ assertion is true then a computation
       loop will check that it is true.
\item  However, if it is false, there is no way to check it in finite 
       time.
\item  Whence the strategy to systematically answer ``NO" to each 
       $\exists x\ldots$ question and then check this answer via 
       a computation loop.
\item  Every false answer ``NO" will be proved false at some time
       during this loop, giving a possibility to correct it.
\item  Every computation using oracle $\emptyset'$ which halts uses
       finitely many questions to the oracle. The above strategy 
       will therefore be corrected only finitely many times so that
       it eventually leads for a fair emulation.
\end{itemize}
Define $h:\prog\to\prog$ such that 
$U^{\cal E}_1(h(p))$ and $U^{\cal E}_2(h(p))$ behave as follows:
\begin{enumerate}
\item  $U^{\cal E}_1(h(p))$ emulates $p$ and whenever $p$ outputs 
       $1$ then $U^{\cal E}_1(h(p))$ will converge on $t$ where $t$ 
       is the current computation step.  
\item  Each time $p$ asks a question to oracle $\emptyset'$ of the 
       form $\exists x\ldots \mbox{ ?}$ then 
       $U^{\cal E}_1(h(p))$ answers ``NO". 
\item  Cautiously, $U^{\cal E}_1(h(p))$ checks each of its oracle 
       answers starting a computation loop which will halt only if 
       the right answer was ``YES" (instead of ``NO").
\item  In case some answer was false, then $U^{\cal E}_1(h(p))$
       restarts the whole emulation of $p$ (correcting its past
       answers) and $U^{\cal E}_2(h(p))$ will converge on a set of 
       points in bijection with that on which $U^{\cal E}_1(h(p))$ 
       was made converging before the answer was recognized to be 
       false. 
\end{enumerate}
Corrections brought in point 4 make the final difference 
$$card^{\nat}(U^{\cal E}_1(h(p))) - card^{\nat}(U^{\cal E}_2(h(p)))$$
equal to $n$.\\
The strong universality of $U^{\cal E}$ insures that 
there exists ${\tt e}$ such that 
$$\forall p\ (U^{\cal E}_1(h(p))
             = U^{\cal E}(\langle {\tt e_1}, p\rangle)\
\wedge\ U^{\cal E}_2(h(p))
             = U^{\cal E}(\langle {\tt e_2}, p\rangle))$$ 
This leads to $K^{\Znat}_{card}(n) \leq K'(n)+2|{\tt e}|+1$,
whence $K^{\Znat}_{card}\segment\nat \leqct K'$.
\end{proof}
%
\subsection{$K_{ord}$ and ${K'}^{\infty}$}
%
\begin{theorem} $\\$
    \centerline{$K_{ord} \eqct {K'}^{\infty}$}
\end{theorem}
\begin{proof}$\\$
Now ${\cal E}=PR(\nat^2 \to \nat)$.
\\
1) Let $n\in\nat$ and let $p\in\prog$ be a minimal length program
such that $ord(U^{\cal E}(p)) = n$.
Define $h:\prog\to\prog$ such that $h(p)$ uses oracle 
$\emptyset'$ and behaves as follows:
\begin{itemize}
\item  $h(p)$ initializes a set $X$ of integers to $\emptyset$.
\item  {\bf Step $0$.} $h(p)$ checks if there is some pair with $0$ 
       as a component which is in the domain of $U^{\cal E}(p)$.
       If so it outputs a $1$ and puts $0$ in the set $X$.
\item  {\bf Step $t>0$.} $h(p)$ checks if there is some chain 
       (constituted of pairs in the domain of $U^{\cal E}(p)$) 
       containing $t$ and all elements of $X$.
       If so it outputs a $1$ and puts $t$ in the set $X$.
\end{itemize}
It is clear that, through an infinite computation, $h(p)$ outputs
the unary representation of the order type of the transitive closure
of the domain of $U^{\cal E}(p)$ in case it is a finite ordered type.
Thus, ${K'}^{\infty}(n) \leq |h(p)|$. 
But there is some constant $c$ such that $|h(p)|= |p|+c$, 
whence ${K'}^{\infty}(n) \leq |p|+c$ and
therefore ${K'}^{\infty} \leqct K_{ord}$.
\medskip\\
2) Let $n\in\nat$, $n>0$ and let $p\in\prog$ be a minimal length 
program using oracle $\emptyset'$ which outputs $n$ in unary 
through an infinite computation. Then ${K'}^{\infty}(n) = |p|$. 
We shall again use Chaitin's harmless overshoot technique.
Define $h:\prog\to\prog$ such that $U^{\cal E}(h(p))$ behaves 
as follows:
\begin{enumerate}
\item  $U^{\cal E}(h(p))$ initializes a set $X$ of integers to 
       $\{0\}$. This set $X$ will always be finite.
\item  $U^{\cal E}(h(p))$ emulates $p$ and whenever $p$ outputs 
       $1$ then a new point $k$ is added to $X$ and 
       $U^{\cal E}(h(p))$ will converge on every pair $(x,k)$ 
       where $x\in X$. 
\item  Each time $p$ asks a question to oracle $\emptyset'$ of the 
       form $\exists x\ldots \mbox{ ?}$ then 
       $U^{\cal E}(h(p))$ answers ``NO". 
\item  Cautiously, $U^{\cal E}(h(p))$ checks each one of its oracle 
       answers starting computation loops which will halt only if 
       some right answer was ``YES" (instead of ``NO").
\item  In case some answer was false, then $U^{\cal E}(h(p))$
       restarts the whole emulation of $p$ (correcting its past
       answers) and $U^{\cal E}(h(p))$ will converge on all
       pairs $(x,y)\in X^2$ and the set $X$ is reinitialized to 
       $\{0\}$. 
\end{enumerate}
Corrections brought in point 5 make the final transitive closure of
the domain of $U^{\cal E}(h(p))$ a preordered set the quotient of 
which has exactly $n$ points.\\
The strong universality of $U^{\cal E}$ insures that 
there exists ${\tt e}$ such that 
$$\forall p\ (U^{\cal E}(h(p))
             = U^{\cal E}(\langle {\tt e}, p\rangle)
$$ 
This leads to $K_{ord}(n) \leq {K'}^{\infty}(n)+2|{\tt e}|+1$,
whence $K_{ord} \leqct {K'}^{\infty}$.
\end{proof}


\begin{thebibliography}{10}

\bibitem{alpha}
V.~Becher, S.~Daicz, and G.~Chaitin.
\newblock A highly random number.
\newblock In {\em Combinatorics, Computability and Logic: Proceedings of the
  Third Discrete Mathematics and Theoretical Computer Science Conference
  (DMTCS'01)}, pages 55--68. Springer-Verlag, 2001.

\bibitem{chaitin75}
G.~Chaitin.
\newblock A theory of program size formally identical to information theory.
\newblock {\em Journal of the ACM}, 22:329--340, 1975.
\newblock Available on Chaitin's home page.

\bibitem{chaitin76}
G.~Chaitin.
\newblock Information theoretic characterizations of infinite strings.
\newblock {\em Theoret. Comput. Sci.}, 2:45--48, 1976.
\newblock Available on Chaitin's home page.

\bibitem{ferbusgrigo}
M.~Ferbus-Zanda and S.~Grigorieff.
\newblock Is randomnes native to computer science?
\newblock {\em Bull. EATCS}, 74:78--118, 2001.
\newblock Available on the author's home page.

\bibitem{kolmo65}
A.N. Kolmogorov.
\newblock Three approaches to the quantitative definition of information.
\newblock {\em Problems Inform. Transmission}, 1(1):1--7, 1965.

\bibitem{levin73}
L.~Levin.
\newblock On the notion of random sequence.
\newblock {\em Soviet Math. Dokl.}, 14(5):1413--1416, 1973.

\bibitem{li-vitanyi}
M.~Li and P.~Vitanyi.
\newblock {\em An introduction to Kolmogorov complexity and its applications}.
\newblock Springer, 1997 (2d edition).

\bibitem{odifreddi}
P.~Odifreddi.
\newblock {\em Classical Recursion Theory}, volume 125.
\newblock North-Holland, 1989.

\bibitem{rogers}
H.~Rogers.
\newblock {\em Theory of recursive functions and effective computability}.
\newblock McGraw-Hill, 1967.

\bibitem{shen}
A.~Shen.
\newblock Kolmogorov complexity and its applications.
\newblock {\em Lecture Notes, Uppsala University, Sweden}, pages 1--23, 2000.
\newblock http://www.csd.uu.se/\~{}vorobyov/ Courses/KC/2000/all.ps.

\bibitem{solovay77}
R.M. Solovay.
\newblock On random {R.E.} sets.
\newblock In A.I. Arruda and al., editors, {\em Non-classical {L}ogics, {M}odel
  theory and {Computability}}, pages 283--307. North-H{o}lland, 1977.

\end{thebibliography}
\end{document}